\documentclass[12pt]{amsart}
\usepackage{hyperref}
\usepackage{amssymb,amsfonts}
\usepackage{hyperref}
\usepackage{amssymb,textcomp}

\usepackage{enumerate}

\usepackage[utf8]{inputenc}
\usepackage[T1]{fontenc}

\usepackage[mathscr]{eucal}

\usepackage[a4paper, twoside=false, vmargin={2cm,3cm}, includehead]{geometry}

\newtheorem{lemma}{Lemma}[section]
\newtheorem{theorem}[lemma]{Theorem}
\newtheorem*{theorem*}{Theorem}
\newtheorem{corollary}[lemma]{Corollary}

\newtheorem{proposition}[lemma]{Proposition}
\newtheorem*{proposition*}{Proposition}

\newtheorem{problem}{Problem}
\newtheorem*{problem*}{Problem}

\theoremstyle{definition}
\newtheorem*{claim*}{Claim}

\newtheorem*{definition}{Definition}

\newtheorem*{remark}{Remark}
\newtheorem*{remarks}{Remarks}

\DeclareMathOperator{\Dlim}{D-lim}
\newcommand{\lE}{\mathbb{E}^{\log}}
\newcommand{\C}{{\mathbb C}}
\newcommand{\E}{{\mathbb E}}
\newcommand{\D}{{\mathbb D}}

\newcommand{\N}{{\mathbb N}}
\renewcommand{\P}{{\mathbb P}}

\newcommand{\R}{{\mathbb R}}
\renewcommand{\S}{\mathbb{S}}
\newcommand{\T}{{\mathbb T}}
\newcommand{\Z}{{\mathbb Z}}
\newcommand{\U}{{\mathbb U}}

\newcommand{\CX}{{\mathcal X}}
\newcommand{\CZ}{{\mathcal Z}}

\newcommand{\bM}{{\mathbf{M}}}

\newcommand{\wh}{\widehat}

\newcommand{\norm}[1]{\left\Vert #1\right\Vert}

\DeclareMathOperator{\id}{id}

\DeclareMathOperator{\reel}{Re}
\renewcommand{\Re}{\reel}

\begin{document}

\title[Correlations of multiplicative functions]{Correlations of multiplicative functions along special sequences}
\title[Correlations of multiplicative functions]{Correlations of multiplicative functions along deterministic and independent sequences}
\thanks{The author was supported  by the
	 Hellenic Foundation for
	Research and Innovation, 
	Project
	No: 1684.}

\author{Nikos Frantzikinakis}
\address[Nikos Frantzikinakis]{University of Crete, Department of mathematics and applied mathematics, Voutes University Campus, Heraklion 71003, Greece} \email{frantzikinakis@gmail.com}
\begin{abstract}
We study correlations of multiplicative functions taken along deterministic sequences and sequences that satisfy certain linear independence assumptions. The results obtained extend recent results of Tao and  Ter\"av\"ainen and results of the author. Our approach is to use tools from ergodic theory in order to effectively exploit feedback from analytic number theory. The results on deterministic sequences crucially use structural properties of measure preserving systems associated with bounded multiplicative functions that were recently obtained by the author and Host. The results on independent sequences  depend on multiple ergodic theorems obtained using the theory of characteristic factors and qualitative equidistribution results on nilmanifolds.
\end{abstract}




\maketitle

\section{Introduction and main results}\label{S:MainResults}
Let $\lambda\colon \N\to \{-1,+1\}$ be the Liouville function that  is defined to be $1$ on integers with an even number of prime factors, counted with  multiplicity, and $-1$ elsewhere. Its values are expected to be randomly distributed and
based on this several conjectures have been formulated. One such conjecture, by Chowla \cite{Ch65}, asserts that the values of $\lambda$ form a normal sequence of $\pm 1$. Equivalently,
if  $n_1,\ldots, n_\ell\in \N$ are distinct, then (see Section~\ref{SS:notation} for our notation on averages)
$$
\E_{m\in\N} \,  \lambda(m+n_1)\cdots \lambda(m+n_\ell)=0.
$$
The conjecture  is settled for $\ell=1$, this is elementarily equivalent to the prime number theorem, and remains open for  $\ell\geq 2$.
Recently, a  version involving logarithmic averages
was established for $\ell=2$ by Tao \cite{T15} and for all odd values of $\ell$ by
Tao and  Ter\"av\"ainen  \cite{TT18b,TT18}. Similar results are also known for Ces\`aro averages on almost all scales~\cite{TT18c}. But even in its logarithmic form, the conjecture  remains open for all even  $\ell\geq 4$.

It is also expected that if $a\colon \N\to \N$ is low complexity strictly increasing sequence, 
 then the
sequence $\lambda\circ a$ inherits the randomness properties of $\lambda$. It  is indeed a classical result  of Kamae~\cite{K73} and Weiss~\cite{W71}, proved in  the 70's,
that normality of a sequence is preserved by composition with a deterministic sequence.
But since normality of the
 Liouville function is  unknown,  it is unclear how to extend  known results about correlations of  $\lambda$ to results about $\lambda\circ a$. Our
first goal is to solve this problem for a large class of deterministic sequences.
 In particular, it is  a consequence of Theorem~\ref{T:deterministic} below,  that if the sequence  $a\colon \N\to \N$ is  deterministic and totally ergodic  (for example,  take $a(n)=[n\alpha+\beta]$ where $\alpha>1$ is  irrational and $\beta\in\R$), then for  $\ell=2$ and for all   odd $\ell$ we have for all distinct  $n_1,\ldots, n_\ell\in \N$ that
$$
\lE_{m\in\N}\, \lambda(a(m+n_1))\cdots \lambda(a(m+n_\ell))=0.
$$

On a slightly different direction, Matom\"aki, Radziwi{\l}{\l},  and Tao \cite{MRT15} established an averaged version of the Chowla conjecture, implying that if $\bM:=(M_k)_{k\in\N}$, with $M_k\to\infty$, is  such that  all limits $\E_{m\in{\bM}}$ below exist,  then
\begin{equation}\label{E:Chowla}
\lim_{N\to\infty}\E_{n_1,\ldots,n_\ell\in [N]} \big|\E_{m\in\bM}\, \lambda(m) \, \lambda(m+n_1)\cdots \lambda(m+n_\ell)\big|=0.
\end{equation}
In \cite{F17} this result was   extended
  to shifts given by  
  arbitrary
   linearly independent polynomials $p_1,\ldots, p_\ell \colon \N^r \to \Z$ with zero constant terms.
Our second goal is to show that if we replace in \eqref{E:Chowla} the average $\E_{m\in\bM}$ with  a
  logarithmic average,  then we can establish results for vastly more general
classes of shifts and we can also replace the limit in density with a regular limit.
For instance, it follows from Theorem~\ref{T:indaper''}  that
if $S$ is a subset of $\N^\ell$ with  independent elements (see definition in Section~\ref{SS:independent}), then
$$
\lim_{|n|\to\infty, n\in S}\big(\lE_{m\in\bM}\, \lambda(m)\,  \lambda(m+n_1)\cdots  \lambda(m+n_\ell)\big)=0
$$
where $n_1,\ldots, n_\ell$ denote the coordinates of  $n\in \N^\ell$.
One corollary of this result is Theorem~\ref{T:indaper}, which implies
 that
if  the sequences  $a_1,\ldots, a_\ell \colon \N \to \N$  have different growth rates and $\bM:=(M_k)_{k\in\N}$, with $M_k\to\infty$, is  such that  all limits $\lE_{m\in{\bM}}$ below exist, then
$$
\lim_{n\to\infty}\big(\lE_{m\in\bM}\, \lambda(m)\,  \lambda(m+a_1(n))\cdots  \lambda(m+a_\ell(n))\big)=0.
$$
 Moreover,
we show that the previous results apply to arbitrary collections of multiplicative functions with values on the complex unit disc  as long as at least one of them satisfies some
aperiodicity assumptions.

 Our last goal is to establish  related results for  correlations of arbitrary multiplicative functions   $f_1,\ldots, f_\ell\colon \N\to[-1,1]$. It follows from  Theorem~\ref{T:indequi} below that if $\alpha_1,\ldots, \alpha_\ell\in \R$ are rationally independent, and   $\bM$ is as before, then
$$
\E_{n\in\N} \big(\lE_{m\in\bM}\,  f_1(m+[n\alpha_1])\cdots  f_\ell(m+[n\alpha_\ell])\big)= \lE_{m\in\bM}f_1(m) \cdots  \lE_{m\in\bM}f_\ell(m).
$$

Note that this identity is no longer true if we replace the average $\E_{n\in\N}$ with a limit or a limit in density, or if the multiplicative functions $f_1,\ldots, f_\ell$ take values on the complex unit disc. On the other hand, we show that the shifts can be replaced by arbitrary collections of sequences that satisfy some linear independence and equidistribution assumptions.

\subsection{Definitions and notation} \label{SS:notation}
In order to  facilitate exposition, we introduce some definitions and notation.
\subsubsection{Averages}
For  $N\in\N$  we let  $[N]:=\{1,\dots,N\}$. Let $a\colon \N\to \C$ be a  sequence.   If $A$ is a non-empty finite subset of $\N$  we let
$$
\E_{n\in A}\,a(n):=\frac{1}{|A|}\sum_{n\in A}\, a(n), \quad
\lE_{n\in A}\,a(n):=\frac{1}{\sum_{n\in A}\frac{1}{n}}\sum_{n\in A}\frac{a(n)}{n}.
$$
If $A$ is an infinite subset of $\N$ we let
$$
\E_{n\in A}\, a(n):=\lim_{N\to\infty} \E_{n\in A\cap [N]}\, a(n), \quad
\lE_{n\in A}\, a(n):=\lim_{N\to\infty} \lE_{n\in A\cap [N]}\, a(n)
$$
if  the limits exist. Also, if $a\colon \N^r\to \C$ is a sequence, we let
$$
\E_{n\in \N^r}\, a(n):=\lim_{N\to\infty} \E_{n\in  [N]^r}\, a(n)
$$
if the limit exists.

Lastly, let $\bM= ([M_k])_{k\in\N}$ be a sequence of intervals with $M_k\to \infty$.
We let
$$
\E_{n\in\bM}\, a(n):=\lim_{k\to\infty}\E_{n\in[M_k]} \, a(n), \quad
\lE_{n\in\bM}\, a(n):=\lim_{k\to\infty}\lE_{n\in[M_k]} \, a(n)
$$
if  the limits exist. Henceforth, we implicitly assume that the lengths of
all  sequences of intervals considered increase to infinity.

\subsubsection{Convergence in density}
A subset $Z$ of $\N^r$ has
{\em (natural) density $0$} if
$$
\lim_{N\to\infty}\frac{|Z\cap [N]^r|}{N^r}=0.
$$ If $n=(n_1,\ldots, n_r)\in \N^r$, we let $|n|:=|n_1|+\cdots+|n_r|$.
\begin{definition}\label{D:UD-lim}
	Let $r\in\N$. We say that the sequence $a\in \ell^\infty(\N^r)$ {\em converges in density} to $0$,
	and write
	$\Dlim_{|n|\to\infty} a(n)=0$, if any of the following three equivalent conditions hold:
	\begin{enumerate}

		\item For every $\varepsilon>0$ the set $\{
		n\in \N^r\colon |a(n)|\geq \varepsilon\}$ has natural density $0$;
		
		\item $\lim_{N\to \infty} \E_{n\in [N]^r}|a(n)|=0$;

\item 	$\lim_{|n|\to \infty, n\notin Z\, }a(n)=0$ for some $Z\subset \N^r$ with natural density $0$.
	\end{enumerate}
\end{definition}

 Several other notions used in the next two subsections are properly defined in Section~\ref{S:prel}.

\subsection{Correlations along deterministic sequences}\label{SS:det}
We start by giving some results related to correlations of multiplicative functions composed with a fixed low complexity sequence.
The following  definition  gives precise meaning to the term low complexity and also defines the notion of total ergodicity that is  crucial for our purposes (see
Section~\ref{S:prel} for basic background in ergodic theory and for the definition of F-systems).
\begin{definition}	With $\U$ we denote the complex unit disc.  We say that:
\begin{itemize}
	\item A sequence $a\colon \N\to  \U$ is {\em totally ergodic} if all its F-systems are totally ergodic.
	It is {\em deterministic} if all its F-systems have zero entropy.

	
\item 	A  sequence $a\colon\N\to\N$ is  {\em totally ergodic} (or {\em deterministic}) if it is strictly increasing,
its range $A:=a(\N)$ is a set of positive density, and
 the $\{0,1\}$-valued sequence ${\bf 1}_A$ is totally ergodic (respectively, deterministic).
	\end{itemize}
	\end{definition}
\begin{remarks}
	$\bullet$ 	In the bibliography (for example in \cite{K73, W71, W00}) a deterministic sequence is often referred to as  {\em completely deterministic}. Note that these definitions use Ces\`aro averages, but for our purposes in the definition of F-systems we use logarithmic averages.
	
	$\bullet$
 	See \cite[Lemma~4.25]{AKLR14} (or \cite{W00}) for necessary and sufficient conditions for a sequence $a\colon \N\to \mathbb{U}$ to be deterministic   that involve the word complexity of
	the  sequence.
	 For
	finite valued sequences,  when deterministic sequences are defined using Ces\`aro averages, they read as follows: For every $\varepsilon>0$ there exists $N\in \N$ such that if we  change the
	values of $a(n)$ on  a set of $n\in \N$ with natural density at most $\varepsilon$, then we can get a new sequence that has at most
	$2^{\varepsilon N}$ words of length $N$ on its range.
\end{remarks}
 An example of a sequence $a\colon \N\to \N$ that  is deterministic  and totally ergodic is $a(n)=n$ and another one is $a(n)=[n\alpha+\beta]$  where $\alpha$ is an irrational greater than $1$ and $\beta\in \R$. Other examples can be given by considering a uniquely and totally ergodic system $(X,\mu,T)$ with zero entropy and constructing $a\colon \N\to \N$ by taking the elements of the set
$S=\{n\in \N\colon T^nx_0\in U\}$ in increasing order, where   $x_0\in X$ is arbitrary and $U$ is a set with positive measure and with boundary of measure zero. One such example is the set    $S:=\{n\in \N\colon \{n^d\alpha\}\in [b,c)\}$ where $d\in \N$, $0\leq b<c<1$,  and   $\alpha$ is irrational (for $d=1$ these examples include all sequences of the form $[n\alpha+\beta]$, where $\alpha>1$ is irrational).

\begin{definition}
	A function $f\colon \N\to \C$ is called \emph{multiplicative} if
$$
f(mn)=f(m)f(n) \ \text{ whenever } \  (m,n)=1.
$$
	It is called \emph{completely multiplicative} if  the previous identity holds for every $m,n\in \N$. A \emph{Dirichlet character} is a periodic
	completely multiplicative function $\chi$  with $\chi(1)=1$.
For convenience,  we extend all multiplicative functions to $\Z$ by letting $f(n)=0$ for $n\leq 0$.
\end{definition}

Henceforth, with $\P$ we denote the set of prime numbers.
For notational convenience we  use  the following  notion of equivalence:
\begin{definition} Let $a,b\colon \P\to \mathbb{U}$.  We write
	$a\sim b$ if  $$
	\lE_{p\in \P}(1-\Re(a(p)\cdot \overline{b(p)}))=0.
	$$
\end{definition}
\begin{remarks}
	$\bullet$	If we restrict to sequences that take values on the unit circle, then $\sim$ is an equivalence relation and $a\sim b$ is equivalent to $\lE_{p\in \P}|a(p)-b(p)|^2=0$.
	
	$\bullet$ Using terminology from  \cite{TT18} we  have  that two multiplicative functions $f,g\colon \N\to\mathbb{U}$ satisfy  $f\sim g$ exactly when   ``$f$ weakly pretends to be $g$''.
\end{remarks}

Our first  theorem extends results of Tao \cite{T15} and Tao, Ter\"av\"ainen~\cite{TT18} that  correspond to the case  $a(n)=n$ (see Section~\ref{S:prel} for definitions of the notions used).

 \begin{theorem}\label{T:deterministic}
 Let   $a\colon \N\to \N$
 be a  deterministic and  totally ergodic sequence. Let $\ell\in \N$ and
 $f_1,\ldots, f_\ell\colon \N\to \U$ be   multiplicative functions such that
 for every Dirichlet character $\chi$ we have 	$f_1\cdots f_\ell\nsim \chi$. Then
 for   	all $n_1,\ldots, n_\ell\in \N$,
  we have
 \begin{equation}\label{E:Odd}
 	\lE_{m\in \N}\prod_{j=1}^\ell f_j(a(m+n_j))=0.
\end{equation}
 Furthermore, the conclusion holds if $\ell=2$, $n_1\neq n_2$, and either $f_1$ or $f_2$ is strongly aperiodic.
 	\end{theorem}
 \begin{remarks}
 	$\bullet$ The conclusion is expected to hold for all deterministic sequences $a\colon \N\to \N$. But even relaxing the total ergodicity assumption to ergodicity seems difficult.

$\bullet$ Note that \eqref{E:Odd} is non-trivial even when $\ell=1$  (this case  is also implicit in \cite{FH19}).
	\end{remarks}
 Note that when $\ell$ is odd or $\ell=2$,  the previous result applies to the case where all the multiplicative functions are equal to  the Liouville or the M\"obius function.

Henceforth, with $\S$ we denote the complex unit circle.
Using the previous result for $\ell=2$ and $f_1=f$, $f_2=\bar{f}$, where $f\colon \N\to \S$ is a strongly aperiodic multiplicative function, we can immediately deduce using an argument from  \cite{T16} (see \cite[Proposition~2.3]{F19} for the needed result)  the following:
\begin{corollary}
Let   $a\colon \N\to \N$
be a  deterministic and  totally ergodic sequence
and $f\colon\N\to \S$ be  a strongly aperiodic multiplicative function.
Then
\begin{equation}\label{E:unbounded}
\sup_{n\in\N}\Big|\sum_{k=1}^nf(a(k))\Big|=+\infty.
\end{equation}
	\end{corollary}
\begin{remark}
The same argument works if $f$ coincides with a strongly aperiodic multiplicative function $f\colon \N\to \U$ on a set with logarithmic density one  and  satisfies $\liminf_{N\to\infty}\lE_{n\in [N]}|f(n)|>0$.
	\end{remark}
When $a(n)=n$, the divergence in  \eqref{E:unbounded} was established  by Tao~\cite{T16} for every completely multiplicative function $f\colon \N\to \S$ and this was a decisive step in his solution of the Erd\"os discrepancy problem.
It seems likely that the conclusion of the corollary also holds   for every completely multiplicative function $f\colon \N\to \S$ but it is not clear how to prove this when $f$ is  not strongly aperiodic.

\subsection{Correlations along independent sequences}\label{SS:independent}
Next, we give results about correlations of multiplicative functions with shifts belonging to sets, or given by  sequences, that  satisfy certain linear independence properties.
\begin{definition}
	We say that a subset $S$ of $\N^\ell$  \emph{has independent elements} if for every non-zero $k\in \Z^\ell$
	the equation $k\cdot n=0$ has only finitely many solutions in $S$.
\end{definition}
\begin{remark}
	The range of the three collections of  sequences  given in  examples (i)-(iii) below form subsets of
	$\N^\ell$ with independent elements. These  are ``thin sets'', but there are also examples of subsets of $\N^\ell$ with independent elements that have density $1$; to see this, using a standard construction, take $Z$ to be  a set of zero density  that contains all but finitely many  elements of each of the sets $\{n\in \N^\ell\colon k\cdot n=0\}$, where $k\in \Z^\ell$ is  non-zero, and let $S:=\N^\ell\setminus Z$.
\end{remark}

\begin{theorem}\label{T:indaper''}
	Let $\ell\in \N$,
	$f_0,\ldots, f_\ell\colon \N\to \U$  be multiplicative functions that admit log-correlations on
	$\bM$, and suppose that at least one of them is strongly aperiodic.
	Furthermore,   let  $S$ be an infinite  subset of $\N^\ell$  that has independent elements.
	We set  $n_0:=0$ and denote the coordinates of $n\in \N^\ell$ with $n_1,\ldots, n_\ell$.
	Then
	\begin{equation}\label{E:inS'}
		\lim_{|n|\to \infty, n\in S}\big(\lE_{m\in\bM}\, \prod_{j=0}^\ell f_j(m+n_j)\big)=0.
	\end{equation}
\end{theorem}
\begin{remarks}
	$\bullet$ Proving \eqref{E:inS'} under the weaker  assumption that $S$ has distinct elements is as hard as Elliott's conjecture
	(a generalization of the Chowla conjecture). In fact, if  \eqref{E:inS'} holds for the set $S:=\{(n_1n,\ldots, n_\ell n), n\in\N\}$ for some specific $n_1,\ldots, n_\ell\in \N$, then  Theorem~\ref{T:Tao2} below gives  that $\lE_{m\in\bM}\, \prod_{j=0}^\ell f_j(m+n_j)=0$.	

$\bullet$	Note that for $\ell$ odd, Theorem~\ref{T:indaper''} is not covered by  Theorem~\ref{T:deterministic} (for $a(n)=n$) because the assumptions on the multiplicative functions in
each result are  different.
\end{remarks}

From Theorem~\ref{T:indaper''} we can easily deduce (see  Section~\ref{S:independent}) a result about correlations of multiplicative functions with shifts given by sequences that satisfy certain independence properties that we define next.
 \begin{definition}	
Let $\ell\in \N$. We say that a collection of  sequences $a_1,\ldots, a_\ell \colon \N^r\to  \N$ 
\begin{itemize}
\item {\em independent}  if for every   $k_1,\ldots, k_\ell\in \Z$, not all of them zero,   we have that $\sum_{j=1}^\ell k_j a_j(n)\neq 0$ for all but finitely many  $n\in \N^r$.

	\item
  {\em weakly independent} if for every   $k_1,\ldots, k_\ell\in \Z$, not all of them zero, we have that
   $\sum_{j=1}^\ell k_j a_j(n)\neq 0$  outside  a set of $n\in\N^r$ with  density zero.
\end{itemize}
\end{definition}
Restricting to the case $r=1$ we can easily verify  that the following collections of sequences are independent:
\begin{enumerate}[(i)]
\item $[n\alpha_1+\beta_1],\ldots, [n\alpha_\ell+\beta_\ell]$, where $\alpha_1,\ldots, \alpha_\ell\in \R_+$  are rationally independent, meaning, all non-trivial integer combinations of the $\alpha_1,\ldots, \alpha_\ell$ are non-zero, and $\beta_1,\ldots, \beta_\ell\in \R$ are arbitrary.

    \item $a_1(n),\ldots,a_\ell(n)$, where $a_1,\ldots, a_\ell\colon \N\to \N$ have different growth rates, meaning, they satisfy  $\lim_{n\to\infty}a_i(n)/a_j(n)=  0$ or $+\infty$ for $i\neq j$.

     \item $p_1(n),\ldots, p_\ell(n)$, where
$p_1, \ldots, p_\ell\colon \N\to \N$ are  linearly independent polynomials.
\end{enumerate}

 If  $\ell,r \in \N$ and $r\geq \ell$, it is easy to verify that any collection
  $L_1,\ldots, L_\ell\colon \N^r\to \N$ of  linearly   independent linear forms is weakly independent,
    but if $r\geq 2$, then  no such collection can be  independent.



\begin{theorem}\label{T:indaper}
	Let $\ell\in \N$,
	$f_0,\ldots, f_\ell\colon \N\to \U$  be multiplicative functions that admit log-correlations on $\bM$, and suppose that at least one of them is strongly aperiodic.
	\begin{enumerate}[(i)]
\item Let $a_0:=0$ and suppose that   $a_1,\ldots,a_\ell \colon \N^r\to \N$ are  independent sequences. Then
	\begin{equation}\label{E:full}
	\lim_{|n|\to \infty}\big(\lE_{m\in\bM}\prod_{j=0}^\ell f_j(m+a_j(n))\big)=0.
	\end{equation}
	
	\item Let $a_0:=0$ and suppose that   $a_1,\ldots,a_\ell \colon \N^r\to \N$ are weakly independent sequences. Then
	\begin{equation}\label{E:zerodens}
	\Dlim_{|n|\to \infty}\big(\lE_{m\in\bM}\prod_{j=0}^\ell f_j(m+a_j(n))\big)=0.
	\end{equation}
	\end{enumerate}
\end{theorem}

Lastly, we give a result regarding correlations of  arbitrary multiplicative functions taking values on the (real) unit interval. For this,  we need to impose an  equidistribution assumption that we define next.
\begin{definition}	
	Let $\ell,r\in \N$. We say that a collection of  sequences $a_1,\ldots, a_\ell \colon \N^r\to  \N$ is
	{\em jointly equidistributed in congruence classes} if
	the sequence $(a_1(n),\ldots, a_\ell(n))_{n\in\N^r}$ is
	equidistributed in congruence classes.
	Equivalently,
	 for all $u\in \N$ and all $k_1,\ldots, k_\ell\in \Z$, not all of them multiples of $u$, we have   $\E_{n\in\N^r}  \,   e^{\frac{2\pi i}{u} \, \sum_{j=1}^\ell k_ja_j(n)}= 0$.
\end{definition}
It is known that the collection of sequences
$[n\alpha_1],\ldots, [n\alpha_\ell]$, where $1, \alpha_1,\ldots, \alpha_\ell\in \R$  are rationally independent,
and the collection   $[n^{c_1}], \ldots, [n^{c_\ell}]$, where $c_1,\ldots, c_\ell\in \R_+\setminus\N$ are distinct, are both
  jointly equidistributed in congruence classes. The same holds for any  collection of   linear forms  $L_1,\ldots, L_\ell\colon \N^\ell\to \N$ as long as   the determinant of their coefficient matrix has absolute value $1$. As remarked before, the previous collections of sequences are also weakly independent.
\begin{theorem}\label{T:indequi}
	Let  $\ell,r\in \N$. Suppose that $a_0:=0$ and    $a_1,\ldots,a_\ell \colon \N^r\to \N$  are
	weakly independent  sequences that are jointly equidistributed in congruence classes. Furthermore, suppose that the multiplicative functions
	$f_0,\ldots, f_\ell\colon \N\to [-1,1]$  admit log-correlations on  $\bM$. Then
	\begin{equation}\label{E:totind}
		\E_{n\in \N^r}\lE_{m\in\bM}\prod_{j=0}^\ell f_j(m+a_j(n))=\prod_{j=1}^\ell\lE_{m\in\bM}f_j(m).
	\end{equation}
\end{theorem}
\begin{remarks}
	$\bullet$ Identity \eqref{E:totind}  fails for complex valued multiplicative functions. For example, if $\ell=r=1$, $f_0(n):=n^i$, $f_1(n):=n^{-i}$, and $a_1(n):=n$, then $\lE_{n\in \N}\, f_0(n)=0$ but
	$\lE_{m\in \N} \, f_0(m)\, f_1(m+a_1(n))=1$ for every $n\in\N$.

	$\bullet$ The multiplicativity assumption is used in an essential way, it is easy to verify that \eqref{E:totind} is not always true if $f_0,\ldots, f_\ell$ are allowed to be arbitrary sequences  taking values in $[-1,1]$, even if these sequences are jointly totally ergodic.
	
	$\bullet$ Examples of Dirichlet characters show that \eqref{E:totind} fails if we replace the average $\E_{n\in \N^r}$ with the limit in density $\Dlim_{|n|\to\infty}$.
\end{remarks}

\subsection{Sign patterns}
Using the results of the previous subsections it is easy to deduce results on sign patterns attained by  multiplicative functions. The next result is a consequence of Theorem~\ref{T:deterministic} and extends \cite[Corollary~1.10(i)]{TT18}, which corresponds to the case where $a(n)=n$ (the result in \cite{TT18} is stated only for $f=\lambda$ but the argument given works in the more general setup of the next theorem).
\begin{theorem}\label{T:sign}
Let   $a\colon \N\to \N$
be a  deterministic and  totally ergodic sequence and $f\colon\N\to \{-1,1\}$ be a multiplicative function such that $f\nsim \chi$ for every Dirichlet character $\chi$ (it is known that then $f$ is also strongly aperiodic). Furthermore, let   $n_1,n_2,\in \N$ be distinct,   and $\epsilon_0,\epsilon_1, \epsilon_2\in \{-1,+1\}$.
Then the logarithmic density of the set
$
\{m\in \N \colon f(a(m))=\epsilon_0,\,  f(a(m+n_1))=\epsilon_1, \, f(a(m+n_2))=\epsilon_2\}
$
is $\frac{1}{8}$.
\end{theorem}
	Also, arguing as in \cite[Corollary~7.2]{TT18} we can deduce from Theorem~\ref{T:deterministic} that if
	 $n_1,n_2, n_3\in \N$ are distinct   and $\epsilon_0,\epsilon_1, \epsilon_2, \epsilon_3\in \{-1,+1\}$ are arbitrary,
then the set 	$$
	\{m\in \N \colon f(a(m))=\epsilon_0,\,  f(a(m+n_1))=\epsilon_1, \, f(a(m+n_2))=\epsilon_2, \, \, f(a(m+n_3))=\epsilon_3 \}$$
	 has positive lower (natural) density.

In the next statement, if $\bM$ is a sequence of intervals and $\Lambda\subset \N$, we define  $$
d_{\bM}(\Lambda):=\lE_{m\in\bM}\, {\bf 1}_{\Lambda}(m),
$$
 (note that we use logarithmic averages) assuming that the limit exists.
\begin{theorem}\label{T:sign1}
Let $\ell, r\in\N$. Let	$f_0,\ldots, f_\ell\colon \N\to \{-1,1\}$  be strongly aperiodic multiplicative functions that admit log-correlations on $\bM$.
Furthermore, if $n\in \N^r$ and $\epsilon_0,\ldots, \epsilon_\ell\in \{-1,+1\}$,
let ${\bf \epsilon}:=(\epsilon_0,\ldots, \epsilon_\ell)$ and
$$
\Lambda_{n,{\bf \epsilon}}:=\{m\in \N \colon f_0(m)=\epsilon_0, f_1(m+a_1(n))=\epsilon_1,\ldots, f_\ell(m+a_\ell(n))=\epsilon_\ell\}.$$
\begin{enumerate}[(i)]
	\item If the sequences  $a_1,\ldots,a_\ell \colon \N^r\to \N$ are  independent, then
$$
\lim_{|n|\to\infty}(d_{\bM}(\Lambda_{n,{\bf \epsilon}}))=2^{-(\ell+1)}.
$$

	\item If the sequences  $a_1,\ldots,a_\ell \colon \N^r\to \N$ are weakly independent, then
	$$
	\Dlim_{|n|\to\infty}(d_{\bf M}(\Lambda_{n,{\bf \epsilon}}))=2^{-(\ell+1)}.
	$$
\end{enumerate}
\end{theorem}
Lastly, we state  an immediate consequence of the previous result.
\begin{corollary}
Let $\ell\in \N$. Suppose that  $a_1,\ldots,a_\ell\colon \N\to \N$ are   sequences with different growth rates, or $a_j(n)=[n\alpha_j]$, $j=1,\ldots, \ell$, where $\alpha_1,\ldots, \alpha_\ell\in \R$ are  rationally independent. Then  for all but finitely many $n\in \N$,  for all
$\epsilon_0,\ldots, \epsilon_\ell\in \{-1,+1\}$, there exist (infinitely many)  $m\in\N$ such that
$$
\lambda(m)=\epsilon_0,\,  \lambda(m+a_1(n))=\epsilon_1,\ldots,\  \lambda(m+a_\ell(n))=\epsilon_\ell.
$$
Furthermore, the same conclusion holds if in place of $\lambda$ we use  any other strongly aperiodic multiplicative function $f\colon \N\to \{-1,1\}$.
\end{corollary}
We remark that a similar  result is unknown for linearly dependent sequences, for example if $\ell=4$ and $a_j(n)=jn$ for $j=1,2,3,4$
(see \cite{TT19} for related progress).

\subsection{Proof strategy} To prove Theorem~\ref{T:deterministic} we make  essential use of a structural result from \cite{FH18, FH19} for measure preserving systems (called Furstenberg systems) naturally associated with arbitrary multiplicative functions $f_1,\ldots, f_\ell\colon \N\to\U$.  This structural result is used implicitly in   the disjointness statement of Theorem~\ref{T:DisjointSeveral} and  gives the identities of Theorem~\ref{T:identity}. These identities allow us to   deduce Theorem~\ref{T:deterministic}
from the main results in \cite{T15} and \cite{TT18}.

To prove Theorems~\ref{T:indaper''} and \ref{T:indequi} we first reinterpret them in ergodic terms using Proposition~\ref{P:correspondence} and then use the  identities of Theorem~\ref{T:Tao2} in order
to reduce matters to proving the ergodic theorems stated in   Propositions~\ref{P:ET2} and \ref{P:ET1} respectively. To prove these
ergodic theorems we use the theory of characteristic factors (see \cite[Chapter~21]{HK18} for a
description of the general method) and qualitative equidistribution results on nilmanifolds.

Lastly, we remark that the number theoretic results of Matom\"aki and  Radziwi{\l}{\l} \cite{MR15} and Matom\"aki,  Radziwi{\l}{\l}, and Tao  \cite{MRT15}, on averages of multiplicative functions in short intervals, are used  in an essential way in all our results  except  the first part of
Theorem~\ref{T:deterministic}.
 In the ergodic setting, this  number theoretic input translates to the fact that a certain function  is orthogonal to the  Kronecker factor of the system (see Part~(i) of Proposition~\ref{P:nt}). Moreover, in  the proof of  Theorem~\ref{T:indequi} we use another fact from  \cite{MR15}, that the mean value of a real valued bounded multiplicative function  is essentially constant on the typical short interval, a property that fails for complex valued multiplicative functions.

\subsection{Some open problems}
In \cite{FH18,FH19} it was shown that any F-system of a collection of multiplicative functions $f_1,\ldots, f_\ell\colon \N\to \U$ has no irrational spectrum. We are unable to prove a similar result
for   collections $f_1\circ a,\ldots, f_\ell\circ a\colon \N\to \U$, where $a\colon \N\to \N$ is a deterministic sequence even if it is totally ergodic. In fact, there is a serious obstacle in proving this even when  $f_1=\cdots=f_\ell=\lambda$. The reason is that it  is consistent with  existing knowledge (though highly unlikely) that
an F-system of $\lambda$ on some sequence of intervals $\bM$  is isomorphic to the system defined by the transformation $T(x,y)=(x,y+x)$ acting on $\T^2$ with the Haar measure. If this is the case,  then  for $a(n)=[n\alpha]$, $n\in\N$, where $\alpha>1$ is irrational,  we can check that $e^{2\pi i \alpha}$ belongs to the spectrum of  the $F$-system of $\lambda\circ a$ on $\bM$. A similar obstacle prevents us from
proving  a variant of  Sarnak's conjecture for $\lambda\circ a$  for ergodic weights, namely that
$$
\lE_{n\in\N}\, \lambda(a(n))\, w(n)=0
$$
whenever $a\colon \N\to\N$ is  deterministic and totally ergodic
and $w\colon \N\to\U$ is  ergodic
 (in \cite{FH18} this was established when   $a(n)=n$).

In any case, the following statement seems plausible and if proved it would solve the problems just mentioned:
\begin{problem}
Let  $a\colon \N\to \N$ be a
 deterministic and  totally ergodic sequence
 and for $\ell\in\N$ let $f_1,\ldots, f_\ell\colon \N\to \U$ be multiplicative functions.
 Then   all the joint correlations of the sequences
      $f_1\circ a,\ldots, f_\ell\circ a$ coincide with the joint correlations of
      $f_1,\ldots, f_\ell$ (and thus the corresponding F-systems coincide).
\end{problem}
\begin{remark}
Equivalently, the conclusion asserts  that
$$
\lE_{m\in\bM} \prod_{j=1}^s g_j(a(m+n_j))=\lE_{m\in\bM} \prod_{j=1}^s g_j(m+n_j)
$$
for all $s\in \N$, $n_1,\ldots, n_s\in \N$, and $g_1,\ldots, g_s\in \{f_1,\ldots, f_\ell, \overline{f}_1,\ldots, \overline{f}_\ell\}$, whenever both limits $\lE_{m\in\bM}$ exist.
\end{remark}

Lastly we mention a problem related both to Theorem~\ref{T:deterministic} and Theorem~\ref{T:indequi}.
\begin{definition}
We say that
a collection of  sequences $a_1,\ldots, a_\ell \colon \N\to  \U$ is {\em jointly totally ergodic} if all its  F-systems are totally ergodic.
\end{definition}
 \begin{problem}
 	Let $\ell\in \N$.
 Suppose that the sequences  $a_1,\ldots,a_\ell \colon \N\to \N$
 are  deterministic, weakly independent,  and jointly totally ergodic.
 If $f_1,\ldots, f_\ell\colon \N\to [-1,1]$ are arbitrary  multiplicative functions,
then
\begin{equation}\label{E:ind}
\lE_{n\in \N}\prod_{j=1}^\ell f_j(a_j(n))=\prod_{j=1}^\ell\lE_{n\in\N}f_j(n).
\end{equation}
 \end{problem}
\begin{remarks}
$\bullet$ A particular case of interest is when $f_j=\lambda$ and $a_j(n)=[n\alpha_j]$, $j=1,\ldots, \ell$, where $1, \alpha_1,\ldots, \alpha_\ell$ are rationally independent real numbers.

$\bullet$ Identity \eqref{E:ind} is false for complex valued multiplicative functions. For example if $\ell=2$, $f_1(n)=n^i$, $f_2(n)=n^{-i}$, and $a_1(n)=[n\alpha]$, $a_2(n)=[n\beta]$ with $1,\alpha,\beta$ rationally independent, then $\lE_{n\in \N}f_1(n)=0$ but
$\lE_{n\in \N} \, f_1(a_1(n))\, f_2(a_2(n))=1$.

\end{remarks}

\section{Background in ergodic theory and number theory}\label{S:prel}

\subsection{Measure preserving systems}\label{SS:mps}
Throughout the article, we make the standard assumption that all probability  spaces $(X,\CX,\mu)$ considered are Lebesgue, meaning, $X$  can be given the structure of a compact metric space
and $\CX$ is its Borel $\sigma$-algebra.
A {\em measure preserving system}, or simply {\em a system}, is a quadruple $(X,\CX,\mu,T)$
where $(X,\CX,\mu)$ is a probability space and $T\colon X\to X$ is an invertible, measurable,  measure preserving transformation. We typically omit the $\sigma$-algebra $\CX$  and write $(X,\mu,T)$. The system is {\em ergodic} if the only sets that
are  invariant by $T$ have measure $0$ or $1$. It  is {\em totally ergodic} if the system $(X,\mu,T^n)$ is ergodic for every $n\in \N$.
 Throughout,  for $n\in \N$ we denote  by $T^n$   the composition $T\circ  \cdots \circ T$ ($n$ times) and let $T^{-n}:=(T^n)^{-1}$ and $T^0:=\id_X$. Also, for $F\in L^1(\mu)$ and $n\in\Z$ we denote by  $T^nF$ the function $F\circ T^n$.

In order to avoid unnecessary repetition,  we refer the reader to the article \cite{FH18} for some other standard notions from ergodic theory.
In particular, the reader will find in Section~2 and
in  Appendix~A of \cite{FH18} the definition
of the terms factor,  conditional expectation with respect to a factor,  (rational) Kronecker factor, isomorphism, inverse limit,  ergodic decomposition, joining, and disjoint systems; all these notions  are used in this article.

\subsection{Furstenberg systems associated with  bounded sequences}\label{SS:Furst}
For the purposes of this  article,
all averages in the definitions below are taken to be logarithmic. The reason is that we  invoke results like Theorems~\ref{T:DisjointSeveral} and \ref{T:Tao2} below that  are only known when the joint Furstenberg systems are defined using logarithmic averages. This limitation comes from the number theoretic input used in their proofs, in particular  the identities in \cite[Theorem~3.1]{FH19} that are based on the entropy decrement argument of Tao~\cite{T15}.
\begin{definition}
	Let $ \bM:=([M_k])_{k\in\N}$ be a sequence of intervals with $M_k\to \infty$.
	We say that a finite collection of bounded sequences $a_1,\ldots, a_\ell\colon \N\to \U$
	{\em admits log-correlations  on $\bM$}, if the   limits
	$$
	\lim_{k\to\infty} \lE_{m\in [M_k]}\,  \prod_{j=1}^s \tilde{a}_j(m+n_j)
	$$
	exist for all $s \in \N$, all   $n_1,\ldots, n_s\in \Z$,
	and all $\tilde{a}_1,\ldots,\tilde{a}_s\in \{a_1,\ldots, a_\ell,\overline{a}_1,\ldots,\overline{a}_\ell\}$.
\end{definition}
\begin{remark}
		Given $a_1,\ldots, a_\ell\colon \Z \to \mathbb{U}$,  using a diagonal argument, we get that every sequence of intervals $\bM=([M_k])_{k\in \N}$
	has a subsequence $\bM'=([M_k'])_{k\in\N}$, such that the sequences  $a_1,\ldots, a_\ell$  admit log-correlations on $\bM'$.
\end{remark}
For  every finite collection of sequences that admits log-correlations on a given sequence of intervals,  we use a  variant of the correspondence principle of Furstenberg~\cite{Fu77, Fu} in order
to associate a measure preserving system that captures the statistical properties of these sequences.
\begin{definition}\label{D:correspondence}
	Let $\ell\in\N$ and $a_1,\ldots, a_\ell\colon \Z\to\mathbb{U}$ be sequences that  admit
	log-correlations  on
	$\bM:=([M_k])_{k\in\N}$. We let  $\mathcal{A}:=\{a_1,\ldots,a_\ell\}$, $X:=(\mathbb{U}^\ell)^\Z$,  $T$ be the shift transformation on $X$,
	and $\mu$ be the weak-star limit of the sequence  $(\lE_{m\in[M_k]} \, \delta_{T^m a})_{k\in \N}$
	where $a:=(a_1,\ldots, a_\ell)$ is thought of as an element of $X$. We call $(X,\mu,T)$
	the {\em joint Furstenberg system
		associated with} $\mathcal{A}$ on  $\bM$, or  simply, the {\em F-system of $\mathcal{A}$ on $\bM$}.
\end{definition}
\begin{remark}
	If we are given sequences $a_1,\ldots, a_\ell\colon \N\to \mathbb{U}$, we extend them to $\Z$ in an arbitrary way; then  the measure $\mu$ will  not depend on the extension.
\end{remark}
We state explicitly  some  useful identities that are implicit in  the previous definition.

\begin{proposition}\label{P:correspondence}
	Let $\ell\in\N$ and  $a_1,\ldots, a_\ell\colon \Z\to \mathbb{U}$ be sequences that  admit
	log-correlations  on
	$\bM:=([M_k])_{k\in\N}$ and let $(X,\mu,T)$ be the corresponding F-system on $\bM$.
	For $j=1,\ldots, \ell$,   consider the functions $F_j\in C(X)$ defined by  $F_j(x):=x_j(0)$,
	where we assume that  $x\in X$ has the form $(x_1(n),\ldots, x_\ell(n))_{n\in\Z}$. Then
	\begin{equation}\label{E:correspondence}
	\lE_{m\in {\bM} }\,  \prod_{j=1}^\ell a_j(m+n_j) =\int  \prod_{j=1}^\ell
	T^{n_j}F_j \, d\mu
	\end{equation}
	for all  $n_1, \ldots, n_\ell\in \Z$,
\end{proposition}

Note that a collection of sequences $a_1,\ldots, a_\ell\colon \Z\to \mathbb{U}$ may have several non-isomorphic  F-systems
depending on which sequence of intervals $\bM$ we use in the evaluation of  their joint correlations. We call any such system {\em an F-system of $a_1,\ldots, a_\ell$}.


\subsection{Multiplicative functions }
We denote by  $\mathcal{M}$ the set of all multiplicative functions $f\colon \N\to \mathbb{U}$, where $\mathbb{U}$ is the complex unit disc.	

We  make  use of the following   notion introduced in \cite{MRT15}:
\begin{definition}\label{D:StronglyAperiodic}
	Let $\D\colon \mathcal{M}\times \mathcal{M}\times \mathbb{R}_+ \to [0,\infty]$ be given by
	$$
	\D(f,g;N)^2:=\sum_{p\in \P\cap [N]} \frac{1}{p}\,\bigl(1-\Re\bigl(f(p) \overline{g(p)}\bigr)\bigr)
	$$
	and $M\colon \mathcal{M}\times \mathbb{N} \to [0,\infty)$ be given by
	$$
	M(f;N):=\min_{|t|\leq N} \D(f, n^{it};N)^2.
	$$
	The multiplicative function $f\in \mathcal{M}$ is
	{\em strongly  aperiodic} (or {\em strongly non-pretentious} using terminology from \cite{GS16})
	if $M(f\cdot\chi;N)\to \infty$ as $N
	\to \infty$ for every Dirichlet character $\chi$.
\end{definition}
It is known that the    M\"obius and the Liouville functions are strongly aperiodic (see \cite{T15}).
More generally, if  $f(p)$ is a nontrivial $d$-th root of unity for all $p\in\P$, then $f$ is strongly aperiodic (see for example  \cite[Corollary~6.2]{F18}).

The hypothesis of strong aperiodicity is useful for our purposes because it gives us access to the ergodic property stated in  Part~(i) of Proposition~\ref{P:nt} below and  also to
 the following result of Tao~\cite[Corollary~1.5]{T15}:
\begin{theorem}\label{T:Tao}
If $f, g\colon \N\to \U$ are multiplicative functions and at least one of them is  strongly aperiodic, then
	$$
	\lE_{m\in \N}\,  f(m)\, g(m+n)=0
	$$
	for every $n\in \N$.
\end{theorem}

\section{Proof of results about deterministic sequences}
 Theorem~\ref{T:deterministic} is a  consequence of the main results in \cite{T15} and \cite{TT18} and the following correlation identities:
\begin{theorem}\label{T:identity} Let $a\colon \N\to \N$ be a deterministic  and totally ergodic
 	sequence and  for $\ell \in \N$ let $f_1,\ldots, f_\ell\colon \N\to \U$
 	be multiplicative functions. Then every sequence of intervals $\bM$ has  a subsequence  $\bM'$
   such  that for every $n_1,\ldots, n_\ell\in \N$ all limits below exist and we have the identity
 $$
\lE_{m\in \bM'}  \prod_{j=1}^\ell f_j(a(m+n_j))=\lE_{n\in\bM'} \lE_{m\in \bM'} \prod_{j=1}^\ell f_j(m+a(n+n_j)).
 $$
 \end{theorem}
\begin{remark}
The identity fails if we do not assume that  $f_1,\ldots, f_\ell$ are  multiplicative. It also fails if we remove the assumption that $a\colon \N\to \N$ is deterministic, or if we replace the total ergodicity assumption with ergodicity.
\end{remark}
\begin{proof}[Proof of Theorem~\ref{T:deterministic} assuming Theorem~\ref{T:identity}]
Let $f_1,\ldots, f_\ell$ be multiplicative functions that satisfy the assumptions of Theorem~\ref{T:deterministic}.
Arguing by contradiction, suppose that \eqref{E:Odd} fails. Then there exist a sequence of intervals $\bM$ and $n_1,\ldots, n_\ell\in \N$ such that
$$
\lE_{m\in \bM}\prod_{j=1}^\ell f_j(a(m+n_j))\neq 0.
$$	
By Theorem~\ref{T:identity} the sequence of intervals $\bM$ has a subsequence $\bM'$ such that all limits below exist and we have
\begin{equation}\label{E:nonzero}
\lE_{n\in\bM'} \lE_{m\in \bM'} \prod_{j=1}^\ell f_j(m+a(n+n_j))\neq 0.
\end{equation}
Using the main result in  \cite{TT18} if  $f_1\cdots f_\ell\nsim \chi$ for every Dirichlet character $\chi$,  and the main result in  \cite{T15} (see Theorem~\ref{T:Tao} in this article) if
 $\ell=2$, $n_1\neq n_2$ (note that then $a(n+n_1)\neq a(n+n_2)$), and either $f_1$ or $f_2$ is strongly aperiodic,  we get that
$$
 \lE_{m\in \bM'} \prod_{j=1}^\ell f_j(m+a(n+n_j))=0
 $$
 for every $n\in \N$. This contradicts \eqref{E:nonzero} and completes the proof.
\end{proof}

So our  goal is to prove Theorem~\ref{T:identity}.

\subsection{Setup}\label{SS:Setup} We  prove  Theorem~\ref{T:identity} via a disjointness argument. Our initial setup
loosely follows the one in  \cite[Lemma~4.1]{K73}. With $\Z_+$ we denote the set of non-negative integers.
 First, we establish a correspondence between strictly increasing sequences $a\colon \Z_+\to \Z_+$ and elements of the sequence space $Y:=\{0,1\}^{\Z_+}$.
Let
\begin{equation}\label{E:Z}
Z:=\{y\in Y\colon \sum_{i=0}^\infty y(i)<+\infty\}, \quad Z^*=Z\setminus \{ {\bf 0 } \},
\end{equation}
where ${\bf 0}$ denotes the element of $Y$ that has all its coordinates $0$.
For $y\in Y$ we let $\tau_y\colon \Z_+\to \Z_+$ be  defined by
$$
\tau_y(n):=\min\{k\in \Z_+\colon \sum_{i=0}^ky(i)=n+1\}
$$
if $y\not\in Z$ and $\tau_y(n)=0$ if $y\in Z$.
If   $a\colon \Z_+\to \Z_+$ is a strictly increasing sequence with range  $A$, then it defines the point
$y_a:={\bf 1}_A\in Y$. On the other hand,  we have $\tau_{y_a}(j)=a(j)$ for every  $j\in \Z_+$,
so the map $y\mapsto \tau_{y}$ sends the point $y_a\in \{0,1\}^{\Z_+}$ (note that $y_a\notin Z$)  to the sequence
$a\colon \Z_+\to\Z_+$.

Next, given sequences $b_1,\ldots, b_\ell\colon \Z\to \U$ and $a\colon \Z_+\to \Z_+$ we want to  reinterpret correlations of the sequences $b_1\circ a,\ldots, b_\ell\circ a$ in dynamical terms.
Let $X:=(\U^\ell)^\Z$ where we think of elements $x$ of $X$ as $\ell$-tuples $(x_1,\ldots, x_\ell)$ with $x_1,\ldots, x_\ell\in \U^\Z$. Let $R,S$ be the shifts on the spaces $X,Y$ correspondingly.
For $n\in \Z_+$ and $j=1,\ldots, \ell$, we define the function  $F_{j,n}\colon  X\times Y\to \U$ by
$$
F_{j,n}(x,y):=
x_j(\tau_y(n))\cdot {\bf 1}_{y(0)=1}, \quad x\in X,\,  y\in Y.
$$

Note that for every $n\in \Z_+$ the function $y\mapsto \tau_y(n)$ is continuous on $Y\setminus Z$. Using this it is easy to verify that  for
every $n\in \Z_+$ and $j\in \{1,\ldots, \ell\}$ we have
\begin{equation}\label{E:cont1}
F_{j,n}(x,y) \text{ is continuous on the set } X\times(Y\setminus Z^*)
\end{equation}
and
\begin{equation}\label{E:cont2}
\text{for every } y\in Y \text{ the function }	x\mapsto F_{j,n}(x,y) \text{ is continuous on } X.
\end{equation}

Let $y\notin Z$. An easy computation shows that  for every  $m\in \N$ and $n\in \Z_+$ we have
$$
\tau_{S^my}(n)+m=\tau_y\big(\sum_{i=0}^{m-1}y(i)+n\big).
$$
Combining the last two identities we get for every $j\in \{1,\ldots,\ell\}$ and every  $m\in \N$ and $n\in \Z_+$ that
\begin{equation}\label{E:id1}
F_{j,n}(R^mx,S^my)=
x_j\big(\tau_y\big(\sum_{i=0}^{m-1}y(i)+n\big)\big)\cdot {\bf 1}_{y(m)=1}.
\end{equation}
Hence,  for every $m\in \N$ and $n_1,\ldots, n_\ell\in \Z_+$ we have that
$$
\prod_{j=1}^\ell F_{j,n_j}(R^mx,S^my)=
\prod_{j=1}^\ell x_j\big(\tau_y\big(\sum_{i=0}^{m-1}y(i)+n_j\big)\big)\cdot {\bf 1}_{y(m)=1}.
$$
Therefore, if we let $k_y(0):=0$ and
$$
k_y(m):=\sum_{j=0}^{m-1}y(j), \quad m\in \N,
$$
we have for every $M\in \N$ and $c(0),\ldots, c(k_y(M))\in \C$  that
\begin{equation}\label{E:id2}
\sum_{m=0}^M c(k_y(m)) \, \prod_{j=1}^\ell F_{j,n_j}(R^mx,S^my)=
\sum_{m=0}^{k_y(M)}c(m)\, \prod_{j=1}^\ell x_j(\tau_y(m+n_j)).
\end{equation}


 Let now $b_1,\ldots, b_\ell\colon \Z\to \U$ be arbitrary  sequences and $a\colon\Z_+\to\Z_+$ be a strictly increasing sequence with range a set of  density $\alpha>0$, or equivalently,
   \begin{equation}\label{E:id3}
   \lim_{m\to\infty}\frac{k_{y_a}(m)}{m}=\alpha>0.
   \end{equation}
Clearly $y_a\notin Z$.     We let $b:=(b_1,\ldots, b_\ell)\in X$. Using that
 $\tau_{y_a}(j)=a(j)$, $j\in \Z_+$, equation  \eqref{E:id3}, and the scale invariance of logarithmic averages,
  we deduce using  \eqref{E:id2}   for $c(m):=\frac{1}{m}$ if $m\in \N$ and $c(0):=0$,  that if for some sequence of intervals $\bM$  the limit $\lE_{m\in \bM}$ on the left hand side below exists, then the same holds for the limit on the right hand side and we have the identity
\begin{equation}\label{E:corid}
\lE_{m\in \bM} \prod_{j=1}^\ell F_{j,n_j}(R^mb,S^my_a)=
\alpha \, \lE_{m\in \bM}\prod_{j=1}^\ell b_j(a(m+n_j)).
\end{equation}
This completes the needed dynamical reinterpretation of the correlations of the sequences $b_1\circ a, \ldots, b_\ell\circ a$ that will be used shortly.

\subsection{Proof of Theorem~\ref{T:identity}}
The first ingredient in the proof of Theorem~\ref{T:identity} is the following result:
  \begin{proposition}\label{P:identity} Let  $a\colon \N\to \N$ be a strictly increasing sequence with range a set  $A$ of positive  density.  Let $\ell\in \N$ and suppose that all the  F-systems  of the sequences $b_1,\ldots, b_\ell\colon \N\to \U$ are disjoint from all the F-systems of ${\bf 1}_{A}$. Then every sequence of intervals $\bM$ has  a subsequence  $\bM'$
   such  that for every $n_1,\ldots, n_\ell\in \N$ all limits below exist and we have the identity
  	$$
  	\lE_{m\in \bM'}  \prod_{j=1}^\ell b_j(a(m+n_j))=\lE_{n\in\bM'} \lE_{m\in \bM'} \prod_{j=1}^\ell b_j(m+a(n+n_j)).
  	$$
  \end{proposition}
  \begin{proof}
We follow the notation established in Section~\ref{SS:Setup}. Using  a diagonal argument we can find a subsequence $\bM'=([M_k'])_{k\in\N}$ of $\bM$, such that  for every $n_1,\ldots, n_\ell\in \N$ all limits below exist and by \eqref{E:corid} we have the identity
\begin{equation}\label{E:a1}
\alpha \, \lE_{m\in \bM'}\prod_{j=1}^\ell b_j(a(m+n_j))=\lE_{m\in \bM'} F(R^mb,S^my_a),
\end{equation}
where $y_a:={\bf 1}_A$, $\alpha>0$ is the density of the set $A$, and
\begin{equation}\label{E:F}
F:=\prod_{j=1}^\ell F_{j,n_j}.
\end{equation}
By passing to a subsequence of $\bM'$, which we denote again by $\bM'$,   we can assume
 that the sequence of measures $(\lE_{m\in [M_k']} \, \delta_{(R^mb,S^my_a)})_{k\in \N}$
converges weak-star to a probability measure $\rho$ on $X\times Y$.  Let $\mu$ and $\nu$ be the marginals of $\rho$ and $R,S$ be the shift transformations  on $X$, $Y$ 
 respectively.   Then $\rho$ is $R\times S$-invariant and
$$
  \mu=\lE_{m\in \bM'}\, \delta_{R^mb},\quad   \nu=\lE_{m\in \bM'}\, \delta_{S^my_a}, \quad \rho=\lE_{m\in \bM'}\, \delta_{(R^mb,S^my_a)}
$$
 where all  implicit limits are  weak-star limits.
 Then   $(X,\mu, R)$ is an $F$-system  of the sequences $b_1,\ldots, b_\ell$ and $(Y,\nu,S)$ is an  F-system of $y_a={\bf 1}_{A}$. By assumption, the two systems are disjoint, hence
 \begin{equation}\label{E:disjoint}
 \rho=\mu\times \nu.
\end{equation}

Next, we claim   that the set $Z^*$, defined in \eqref{E:Z}, satisfies $\nu(Z^*)=0$.
Indeed, note that for $y\in Z^*$ the sets $T^{-n}\{y\}$, $n\in\N$, are disjoint.
Using the shift invariance of the probability measure $\nu$ we deduce that $\nu(\{y\})=0$ for every $y\in Z^*$,
and since $Z^*$ is a countable set, we conclude that $\nu(Z^*)=0$.

 Since $\nu(Z^*)=0$ and by  \eqref{E:cont1}  the function $F$ is continuous on $X\times (Y\setminus  Z^*)$ and $\rho$ is a joining of $\mu$ and $\nu$,   the function $F$ is continuous  $\rho$-almost everywhere.
Hence,
$$
\lE_{m\in \bM'} F(R^mb,S^my_a)=\int F(x,y) \, d\rho(x,y)=\int \Big(\int F(x,y)\, d\mu(x)\Big)\, d\nu(y)
$$
where the last identity follows from  \eqref{E:disjoint} and Fubini's theorem.
Since by  \eqref{E:cont1}  the function $F$ is continuous on $X\times (Y\setminus  Z^*)$, we get using  the bounded  convergence theorem  that the function $G\colon Y\to \C$ defined by $G(y):=\int F(x,y)\, d\mu(x)$, $y\in Y$,  is continuous on $Y\setminus Z^*$. Since $\nu(Z^*)=0$, the function  $G$ is continuous  $\nu$-almost everywhere, so we have
$$
\int \Big(\int F(x,y)\, d\mu(x)\Big)\, d\nu(y)=\lE_{n\in\bM'}G(S^ny_a)=\lE_{n\in\bM'}\int F(x,S^ny_a)\, d\mu(x).
$$
Moreover,  by  \eqref{E:cont2}, for every fixed $y\in Y$ the function $x\mapsto F(x,y)$ is continuous on $X$, hence for every $n\in \N$ we have
$$
\int F(x,S^ny_a)\, d\mu(x)=\lE_{m\in\bM'}F(R^m b,S^ny_a)=
\lE_{m\in\bM'}F(R^{n+m} b,S^ny_a).
$$
Combining the last three identities we get
\begin{equation}\label{E:a2}
\lE_{m\in \bM'} F(R^mb,S^my_a)=\lE_{n\in\bM'}\lE_{m\in\bM'}F(R^{n+m}b,S^ny_a).
\end{equation}
Using the definition of the function $F$ in  \eqref{E:F}, identity \eqref{E:id1}, and the fact that $\tau_{y_a}(j)=a(j)$, $j\in \Z_+$,   we get for every $m, n\in\N$ that
\begin{equation}\label{E:id3'}
F(R^{n+m}b,S^n y_a)=\prod_{j=1}^\ell b_j\big(m+a\big(\sum_{i=0}^{n-1}y_a(i)+n_j\big)\big)\cdot {\bf 1}_{y_a(n)=1}.
\end{equation}
Using \eqref{E:id3'} and arguing exactly as in the last part of Section~\ref{SS:Setup} we deduce that
\begin{equation}\label{E:a3}
\lE_{n\in\bM'}\lE_{m\in\bM'}F(R^{n+m}b,S^ny_a)=\alpha \,  \lE_{n\in\bM'} \lE_{m\in \bM'} \prod_{j=1}^\ell b_j(m+a(n+n_j)).
\end{equation}
Combining \eqref{E:a1}, \eqref{E:a2},  \eqref{E:a3} (and using that $\alpha \neq 0$), we get the asserted identity.
\end{proof}


The next result is a crucial element in the proof of Theorem~\ref{T:deterministic} and follows by  combining  the structural result of \cite[Theorem~1.5]{FH19} with the disjointness statement of \cite[Proposition~3.12]{FH18}.
\begin{theorem} \label{T:DisjointSeveral}
	All  F-systems of any bounded collection of  multiplicative functions with values on the complex unit disc  are disjoint from all zero entropy  totally ergodic systems.
\end{theorem}

Combining the previous two results we can now prove Theorem~\ref{T:identity}.
\begin{proof}[Proof of Theorem~\ref{T:identity}] By assumption,  the sequence $a\colon\N\to \N$ is strictly increasing and its range, which we denote by $A$, has positive  density.
Also  by assumption, all F-systems of  ${\bf 1}_A$ have zero entropy and are  totally ergodic.  It follows from Theorem~\ref{T:DisjointSeveral} that
	 all  F-systems  of the collection  $f_1,\ldots, f_\ell\colon \N\to \U$ are disjoint from all F-systems of ${\bf 1}_{A}$. Hence, Proposition~\ref{P:identity} applies and gives that the conclusion of Theorem~\ref{T:identity} holds.
\end{proof}

\section{Proof of results about independent sequences}\label{S:independent}

 In this section we prove Theorems~\ref{T:indaper''}, \ref{T:indaper},  \ref{T:indequi}.

We start with a reduction, we show that   Theorem~\ref{T:indaper} follows from the following more general result:
\begin{theorem}\label{T:indaper'}
	Let $\ell\in \N$. Let
	$f_0,\ldots, f_\ell\colon \N\to \U$  be multiplicative functions that admit log-correlations on $\bM$ and suppose that at least one of them is strongly aperiodic.
Furthermore, let $a_0:=0$  and   $a_1,\ldots,a_\ell \colon \N^r\to \N$ be sequences, and  $R$ be an infinite subset of $\N^r$  such that the set $S:=\{(a_1(n),\ldots, a_\ell(n))\colon n\in R\}$ has independent elements (see definition in Section~\ref{SS:independent}).
Then
	\begin{equation}\label{E:inS}
\lim_{|n|\to \infty, n\in R}\big(\lE_{m\in\bM}\prod_{j=0}^\ell f_j(m+a_j(n))\big)=0.
\end{equation}
\end{theorem}
\begin{proof}[Proof of Theorem~\ref{T:indaper} assuming Theorem~\ref{T:indaper'}.]
Part~$(i)$ of Theorem~\ref{T:indaper} follows at once, since for  independent sequences $a_1,\ldots, a_\ell$, for $R:=\N^r$,  the set $S$ in the statement of Theorem~\ref{T:indaper'}  has independent elements.
 Hence, \eqref{E:inS}  holds with $R=\N^r$,
so \eqref{E:full} holds.

We prove  Part~$(ii)$ of Theorem~\ref{T:indaper}.
Since the sequences $a_1,\ldots,a_\ell$ are weakly independent  all the sets $Z_{k_1,\ldots,k_\ell}:=\{n\in\N^r\colon k_1a_1(n)+\cdots +k_\ell a_\ell(n)=0\}$, $k_1,\ldots, k_\ell\in \Z$,  have zero density unless $k_1=\cdots=k_\ell=0$. Since the collection of all these sets is countable,   it is well known that there exists a subset $Z$ of $\N^r$ that has zero density and such that $Z_{k_1,\ldots, k_\ell}\setminus Z$ is finite  for all $k_1,\ldots,k_\ell\in \Z$ not all of them $0$.
 Then for  $R:=\N^r\setminus Z$
 the set $S$ in the statement of Theorem~\ref{T:indaper'}  has independent elements. Hence, \eqref{E:inS} holds for this set $R$, and since the complement of  $R$ has density zero, we get \eqref{E:zerodens}.
\end{proof}
Next we show that   Theorem~\ref{T:indaper'} follows from Theorem~\ref{T:indaper''}.
\begin{proof}[Proof of Theorem~\ref{T:indaper'} assuming Theorem~\ref{T:indaper''}.]
  If $\ell=1$, then the result follows from Theorem~\ref{T:Tao}. So we can assume that $\ell\geq 2$. In this case  we have  that $\lim_{n\to\infty, n\in R}|(a_1(n),\ldots, a_\ell(n))|=\infty$, because otherwise
	for $n\in R$ the vectors $(a_1(n),\ldots, a_\ell(n))$  attain some fixed value infinitely often,  and this easily contradicts our assumption that the set $S$ in the statement of Theorem~\ref{T:indaper'} has independent elements. Using this, we deduce  \eqref{E:inS}  from  \eqref{E:inS'}.
\end{proof}

Hence, it remains to prove Theorems~\ref{T:indaper''} and \ref{T:indequi}.

\subsection{Ergodic feedback from number theory}
We start by translating some input from number theory to useful ergodic properties.
\begin{proposition}\label{P:nt}
Let $\ell\in \N$. Let $(X,\mu,T)$ be an F-system  on $\bM$ of the multiplicative functions $f_1, \ldots, f_\ell\colon \N\to \U$ and  $F_1,\ldots, F_\ell$ be the functions of Proposition~\ref{P:correspondence}.
\begin{enumerate}[(i)]
\item \label{i} If $f_1$ is strongly aperiodic, then the function $F_1$ is orthogonal to the Kronecker factor of the system, meaning, it is orthogonal to all eigenfunctions of the system.

\item \label{ii} If $f_1$ is real valued, then the function $F_1$  satisfies $\E(F_1|\mathcal{I})=\int F_1\, d\mu$, where $\mathcal{I}:=\{A\in \CX\colon T^{-1}A=A\}$.
\end{enumerate}
\end{proposition}
\begin{remark}
Part \eqref{ii} fails for complex valued multiplicative functions. For example, if $f_1(n):=n^{i}$, $n\in \N$, then it can be shown that $\E(F_1|\mathcal{I})=F_1$ but $\int F_1\, d\mu=0$.
\end{remark}
\begin{proof}
For  \eqref{i}  the  number theoretic input needed  is \cite[Theorem~B.1]{MRT15} and the deduction can be found in  the proof of \cite[Proposition~5.1]{F18}.

For \eqref{ii} the  number theoretic input needed is \cite[Theorem~1]{MR15}.
It implies that
$$
\lim_{N\to\infty}\limsup_{M\to\infty}\E_{m\in [M]}\big|\E_{n\in[N]}f_1(n+m)-\alpha|^2=0
$$
where $\alpha:=\E_{n\in\N}f_1(n)$ (the limit is known to exist by a result of Wirsing~\cite{W67}).
This implies that
$$
\lim_{N\to\infty}\lE_{m\in \bM}\big|\E_{n\in[N]}f_1(n+m)-\alpha|^2=0,
$$
where one sees that the limits  $\lE_{m\in \bM}$ exist by expanding the square and using our assumption that $f_1$ admits log-correlations on $\bM$.
   Expanding the square and using Proposition~\ref{P:correspondence} we get that
$$
\lim_{N\to\infty} \int\big|\E_{n\in[N]}F_1(T^nx)-\alpha|^2\, d\mu=0.
$$
Using the mean ergodic theorem we deduce that
$$
\int |\E(F_1|\mathcal{I})-\alpha|^2\, d\mu=0.
$$
Hence, $\E(F_1|\mathcal{I})=\alpha=\int F_1\, d\mu$.
\end{proof}

\subsection{Reduction to ergodic statements}
In this subsection we  show that Theorems~\ref{T:indaper''} and \ref{T:indequi}
follow from two ergodic statements that we prove subsequently.
In order to carry out the needed reduction we will make crucial use of certain  identities satisfied by F-systems of multiplicative functions. They are based on   work in \cite{T15}  and \cite{TT18} and are proved in \cite[Theorem~3.8]{FH19}. Henceforth, for $d\in \N$ we let $\P_d:=\P\cap(d\Z+1)$.
\begin{theorem}\label{T:Tao2}
	Let $\ell\in \N$ and  $f_1,\ldots, f_\ell\colon \Z\to \mathbb{U}$  be multiplicative functions. There exists $d\in \N$ such that the following holds: If
	$(X,\mu,T)$ is an F-system  of $f_1,\ldots, f_\ell$ and if    $F_1, \ldots, F_\ell$ are as in Proposition~\ref{P:correspondence}, then we have
	\begin{equation}
	\label{eq:Furstenberg-Tao}
	\int \prod_{j=1}^{\ell} T^{n_j} F_j \, d\mu=   \E_{p\in \P_d}  \int \prod_{j=1}^{\ell} T^{pn_j}F_j\, d\mu
	\end{equation}
	for all  $n_1,\ldots, n_{\ell}\in \Z$.
\end{theorem}
We will show that Theorem~\ref{T:indaper''} follows from the following result:
\begin{proposition}\label{P:ET2}
	Let  $\ell\in \N$ and 	 $S$ be a subset of $\N^\ell$ with independent elements.
	Suppose that  $(X,\mu,T)$ is a system and  $F_0,\ldots, F_\ell\in L^\infty(\mu)$ are functions at least one of which is orthogonal to the rational Kronecker factor of the system. 	We set  $n_0:=0$ and denote the coordinates of $n\in \N^\ell$ by $n_1,\ldots, n_\ell$.
	Then for every $d\in \N$ all limits $\E_{p\in\P_d}$ below exist and  we have
	\begin{equation}\label{E:l}
	\lim_{|n|\to\infty, n \in S} \Big(\E_{p\in\P_d} \int  \prod_{j=0}^\ell T^{pn_j}F_j\, d\mu\Big)=0.
	\end{equation}
\end{proposition}
\begin{remark}
A function is orthogonal to the rational Kronecker factor of a system $(X,\mu,T)$ if it is orthogonal to  any function $F\in L^\infty(\mu)$  that satisfies $TF=e^{2\pi i \alpha} \, F$ for some  $\alpha \in \mathbb{Q}$.
\end{remark}


We will show that Theorem~\ref{T:indequi} follows from the following result:

\begin{proposition}\label{P:ET1}
	Let  $\ell, r\in\N$. Let also  $a_0:=0$ and suppose that $a_1,\ldots,a_\ell \colon \N^r\to \N$
	are weakly independent  sequences that are  jointly equidistributed in congruence classes. Then for every $d\in \N$, for all ergodic systems $(X,\mu,T)$ and functions  $F_0,\ldots, F_\ell\in L^\infty(\mu)$,
 all the  limits  below exist and we have
	\begin{equation}\label{E:ET1}
		\E_{n\in\N^r}\E_{p\in\P_d} \int \prod_{j=0}^\ell T^{pa_j(n)}F_j\, d\mu=\prod_{j=0}^\ell \int F_j\, d\mu.
	\end{equation}
\end{proposition}
\begin{remark}
	Examples of periodic systems show that \eqref{E:ET1} fails if we replace $\E_{n\in\N^r}$ with $\Dlim_{|n|\to\infty}$.
\end{remark}

\begin{proof}[Proof of Theorems~\ref{T:indaper''} and \ref{T:indequi} assuming Propositions~\ref{P:ET2} and \ref{P:ET1}]
First we prove Theorem~\ref{T:indaper''} assuming Proposition~\ref{P:ET2}. Let $(X,\mu,T)$ be the  $F$-system of  $f_0,\ldots, f_\ell$ on $\bM$.  By Theorem~\ref{T:Tao2} there exist $d\in \N$ and  functions $F_0,\ldots, F_\ell\in L^\infty(\mu)$ such that  for all $n_0,\ldots, n_\ell\in\Z$ we have
\begin{equation}\label{E:pd}
\lE_{m\in\bM}\prod_{j=0}^\ell f_j(m+n_j)=\E_{p\in \P_d}  \int \prod_{j=0}^{\ell} T^{pn_j}F_j\, d\mu.
\end{equation}

Suppose that $f_{j_0}$ is strongly aperiodic for some $j_0\in \{0,\ldots, \ell\}$. Then by Part~\eqref{i} of Proposition~\ref{P:nt} the function $F_{j_0}$ is orthogonal to the  Kronecker factor of the system $(X,\mu,T)$. Using this and identity \eqref{E:pd},
we deduce from Proposition~\ref{P:ET2} that Theorem~\ref{T:indaper''}  holds.

Next we prove Theorem~\ref{T:indequi}  assuming Proposition~\ref{P:ET1}.
Let $\mu=\int \mu_x\, d\mu$ be the ergodic decomposition of the measure $\mu$.
 We have that
$$
 \E_{n\in\N^r}\E_{p\in\P_d} \int \prod_{j=0}^\ell T^{pa_j(n)}F_j\, d\mu=
\int\Big( \E_{n\in\N^r}\E_{p\in\P_d} \int \prod_{j=0}^\ell T^{pa_j(n)}F_j\, d\mu_x\Big)\, d\mu,
 $$
 where   we used the bounded convergence theorem twice, the first time we used that the limits $\E_{p\in\P_d} \int \prod_{j=0}^\ell T^{pa_j(n)}F_j\, d\mu_x$ exist for every $n\in \N$ (see Theorem~\ref{T:FHK} below) and the second time we used that the limits $\E_{n\in\N^r}\big(\E_{p\in\P_d} \int \prod_{j=0}^\ell T^{pa_j(n)}F_j\, d\mu_x\big)$  exist by Proposition~\ref{P:ET1}. Using Proposition~\ref{P:ET1} once more, we get that
 $$
\E_{n\in\N^r} \E_{p\in\P_d} \int \prod_{j=0}^\ell T^{pa_j(n)}F_j\, d\mu=   \int \Big(\prod_{j=0}^\ell  \int F_j\, d\mu_x\Big)\, d\mu= \int \prod_{j=0}^\ell \E(F_j|\mathcal{I})\, d\mu,
 $$
 where we used that for every $F\in L^\infty(\mu)$ for $\mu$ almost every $x\in X$ we have that $\E(F|\mathcal{I})(x)=\int F\, d\mu_x$ (see for example \cite[Page~37]{HK18}).

By Part \eqref{ii} of Proposition~\ref{P:nt}
we have that $\E(F_j|\mathcal{I})=\int F_j\, d\mu$ for $j=0,\ldots,\ell$
(here we made crucial use of the fact that the multiplicative functions are real valued). Combining the above with identity \eqref{E:pd}  and the fact that (by \eqref{E:correspondence})
$\int F_j\, d\mu=\lE_{m\in\bM}\, f_j(m)$ for $j=1,\ldots,\ell$,
we get that Theorem~\ref{T:indequi} holds.
\end{proof}

Thus, it remains to prove the ergodic statements of  Propositions~\ref{P:ET2} and \ref{P:ET1}. We do this in the remaining  subsections.

\subsection{Nilsystems, nilcharacters, and nilfactors}\label{SS:nil}
If $G$ is a group  we let
$G_1:=G$ and $G_{j+1}:=[G,G_j]$, $j\in \mathbb{N}$. We say that $G$ is {\em $s$-step nilpotent} if  $G_{s+1}$ is the trivial
group.
An {\em $s$-step nilmanifold} is a homogeneous space $X=G/\Gamma$, where  $G$ is an $s$-step nilpotent Lie group
and $\Gamma$ is a discrete cocompact subgroup of $G$. With $e_X$ we denote  the image in $X$ of the unit element of $G$.
An {\it $s$-step  nilsystem} is
a system of the form $(X, \CX, m_X, T_b)$,  where $X=G/\Gamma$ is an $s$-step nilmanifold, $b\in G$,  $T_b\colon X\to X$ is defined by  $T_b(g\cdot e_X) \mathrel{\mathop:}= (bg)\cdot e_X$ for  $g\in G$,  $m_X$  is
the normalized Haar measure on $X$, and  $\CX$  is the completion
of the Borel $\sigma$-algebra of $G/\Gamma$.
We call the map $T_b$ or the element $b$ a {\em nilrotation}. If $T_b$ acts ergodically we call $b$ an {\em ergodic nilrotation}.

With  $G^0$ we denote the connected component of the identity element in $G$. When we are working  with an ergodic nilsystem, we can assume that the space $X$ is represented as  $X=G/\Gamma$ where
$G=\langle G^0,b \rangle$ and $\Gamma$ does not contains  any non-trivial normal subgroups of $G$
(see pages 100 and 177 in \cite{HK18} or \cite[Section 4.1]{BHK05}). Henceforth, we are going to use these properties without further reference.  When we work with  such a representation we have that  for $j\geq 2$   the  commutator subgroups $G_j$ are connected  (see \cite[Page~155]{HK18} or \cite[Theorem~4.1]{BHK05}).
Hence, in an $s$-step nilmanifold with  $s\geq 2$, the subgroup $G_s$ is connected and the Abelian group $K_s:=G_s/(G_s\cap\Gamma)$ is a finite dimensional torus (perhaps the trivial one). Let $\wh{K_s}$ be the dual group of $K_s$; it consists of the characters of $G_s$ that  are $(\Gamma\cap G_s)$-invariant.
An \emph{$s$-step nilcharacter} of $X$ (often called a {\em vertical nilcharacter}) with \emph{frequency} $\chi$,
where   $\chi\in\wh{K_s}$,
is a
function $\Phi\in \C(X)$   that  satisfies
\begin{equation}\label{E:non}
\Phi(u\cdot x)=\chi(u)\, \Phi(x), \ \text{ for every } \ u\in G_s \ \text{  and }  \ x\in X.
\end{equation}
If $\chi$ is a non-trivial character of $K_s$,  we say that  $\Phi$
is a \emph{non-trivial $s$-step nilcharacter},  otherwise we say that it is  a \emph{trivial $s$-step nilcharacter}. It follows from \eqref{E:non} that every non-trivial $s$-step nilcharacter has zero integral. It is also known that
the linear span of $s$-step nilcharacters  is dense in $C(X)$ with the uniform norm
(see for example \cite[Proof of Lemma~2.7]{GT12}).

If the nilmanifold $X$ is not connected, let  $X^0$ be the connected component of $e_X$ in $X$. Then for $s\geq 2$ the restriction of a non-trivial $s$-step nilcharacter $\Phi$ of $X$ onto $X^0$ is a non-trivial $s$-step nilcharacter of $X^0$ with  the same frequency (see \cite[Section~3.3]{F17}).

Let $(X,\mu,T)$ be an ergodic system and for  $k\in \N$ let
$(Z_k, \CZ_k, \mu_k,T)$ be the factor of order $k$ of $X$ as defined in~\cite[Chapter~9]{HK18} (we abuse notation and denote the transformation on $Z_k$ by $T$).
The following  result was  proved in \cite{HK05}:
\begin{theorem}\label{T:HK}
	If $(X,\mu,T)$ is an ergodic system, then for every $k\in \N$ the system $(Z_k,\CZ_k, \mu_k,T)$ is an inverse limit of ergodic $k$-step nilsystems.
\end{theorem}
We remark that  properties of inverse limits imply that if $(X,\mu,T)$ is an ergodic system, then for every $F\in L^\infty(Z_k,\mu_k)$  and  every $\varepsilon>0$, there exist a  $k$-step nilsystem $(X',m_{X'},T')$, a factor map $\pi\colon Z_k\to X'$,  and a function $F'\in L^\infty(m_{X'})$ (in fact, we can take  $F':=\E(F|X')$), such that  $\norm{F-F'\circ\pi}_{L^1(\mu)}\leq \varepsilon$.
Furthermore, the function  $F'$ can be chosen so that if $F$ is orthogonal to the (rational) Kronecker factor of the system $(Z_k, \CZ_k, \mu_k,T)$, then
$F'$ is orthogonal to the (rational) Kronecker factor of the system $(X',m_{X'},T')$.

\subsection{Reduction to statements about nilsystems}
To carry out our reductions further we  will use
the following convergence result:
\begin{theorem}
	\label{T:FHK}
	Let $(X,\mu,T)$ be a system and
	  $d, \ell \in \N$. Then
	  for every  $F_1,\ldots ,F_\ell\in L^\infty(\mu)$,  the following limit exists in $L^2(\mu)$
	$$
	\E_{p\in\P_d} \prod_{j=1}^\ell T^{pj}F_j\, d\mu.
	$$
	Furthermore, the factor  $Z_\ell$ (defined in Section~\ref{SS:nil}) is characteristic for mean convergence of these averages, meaning, if $\E(F_j|Z_\ell)=0$ for some $j\in \{1,\ldots, \ell\}$, then the  averages converge to $0$ in $L^2(\mu)$.
\end{theorem}
\begin{remark}
	We note that for $\ell\geq 2$ the smaller factor $Z_{\ell-1}$ is characteristic for convergence of the previous averages, but we will not need this.
	\end{remark}
 We remark that for $d=1$ the convergence part of this result follows
	 from   \cite{WZ11} and the part about characteristic factors from  \cite{FHK}
	 (conditional to some conjectures obtained later in \cite{GT10, GTZ12}).
	The statement for general $d\in \N$  follows by using the $d=1$ case for product systems of the form
	$T\times R$ acting on $X\times \Z/(d\Z)$ with the product measure,
	where $R$ is the shift on $\Z/(d\Z)$, and for the functions
	$F_1\otimes {\bf 1}_{d\Z+1}, F_2,\ldots, F_\ell$.
	 For the statement on characteristic factors one also uses the fact that  for every  $\ell\in \N$ if $\E(F|Z_\ell(T))=0$, then also  $\E(F|Z_\ell(T\times R))=0$
	 and $\E(F\otimes {\bf 1}_{d\Z+1}|Z_\ell(T\times R))=0$.

\begin{proposition}\label{P:red1}
	If Proposition~\ref{P:ET2} holds for every ergodic nilsystem, then it holds for every  system.
	\end{proposition}
\begin{proof}
First we use an ergodic decomposition argument   to show that it suffices to  verify the statement for ergodic systems. The proof of this reduction is
the same
as the one used in the proof of Theorem~\ref{T:indequi} assuming Proposition~\ref{P:ET1}.
The only additional ingredient needed is the well known fact that if $\mu=\int\mu_x\, d\mu$ is the ergodic decomposition of the measure $\mu$ and a function $F$ is orthogonal to the rational Kronecker factor of the system $(X,\mu,T)$, then for $\mu$ almost every $x\in X$ the function $F$ is orthogonal
to the rational Kronecker factor of the system $(X,\mu_x,T)$.

Next, note  that since the set $S$ has independent elements, if $\ell\geq 2$, then  all but finitely many  $n\in S$ have distinct coordinates.
Hence,  Theorem~\ref{T:FHK}  applies, and implies that in proving \eqref{E:l} we can assume   that the system we work with is
$(Z_k,  \mu_k,T)$ for some $k\in\N$. Next, using
  Theorem~\ref{T:HK}
 and the approximation property mentioned immediately after this result, we get that in proving Proposition~\ref{P:ET2}
  we can assume that the system $(X,\mu,T)$ is an ergodic  $k$-step nilsystem.
\end{proof}

A very similar argument gives the following reduction.
\begin{proposition}\label{P:red2}
	If Proposition~\ref{P:ET1} holds for every ergodic nilsystem, then it holds for every ergodic system.
\end{proposition}
\begin{proof}
First note  that since the collection of sequences $a_1,\ldots, a_\ell$ is weakly independent,  if  $\ell\geq 2$, then for  all $n\in \N^r$ outside a set of density $0$
	the values  $a_1(n),\ldots, a_\ell(n)$ are distinct. Hence,  Theorem~\ref{T:FHK}   applies  and shows that  in proving \eqref{E:ET1} we can assume  that the system we work with is
	$(Z_k,  \mu_k,T)$ for some $k\in \N$. We conclude the reduction as in the previous proposition.
	\end{proof}
	
\subsection{Proof of Propositions~\ref{P:ET2} and \ref{P:ET1}}
Our plan is to first prove Propositions~\ref{P:ET2} and \ref{P:ET1} in the case where the nilmanifold is Abelian and subsequently deal with the  non-Abelian case.
\begin{lemma}\label{L:Abelian}
Propositions~ \ref{P:ET2} and \ref{P:ET1} hold if $(X,\mu,T)$ is a rotation  on a compact Abelian Lie group with the Haar measure.
\end{lemma}
\begin{proof}
	Suppose that $T$ is an ergodic rotation on a compact  Abelian Lie group $X$. Then  $X=\Z_u\times \T^v$ for some $u\in \N$ and $v\in \Z_+$, where $\Z_u:=\Z/(u\Z)$. Moreover, we can assume that
	$$
	T(x,y)=(x+1,y+\alpha), \quad
x\in \Z_u,\,  y\in \T^v,
$$   where addition is taken $\! \! \mod{u}$ on the first coordinate, and $\alpha=(\alpha_1,\ldots, \alpha_v)$  acts ergodically on $\T^v$, or equivalently, $1,\alpha_1,\ldots, \alpha_v$ are rationally independent. Note also that    $\mu=m_X=m_{\Z_u}\times m_{\T^v}$.

First, we prove  Proposition~\ref{P:ET2}.
By approximation in $L^2(\mu)$  it suffices to verify \eqref{E:l} when $$
F_j(x,y):=e\big(k_j\frac{x}{u}\big) \, e(l_j\cdot y),\quad  k_j\in\Z,\,  l_j\in\Z^v, \, j=0,\ldots, \ell,
$$ where $x\in \Z_u, y\in \T^v,$ and $e(t):=e^{2\pi i t}$ for $t\in \R$. Furthermore, since  at least one of the functions $F_0,\ldots, F_\ell$  is orthogonal to the rational Kronecker factor of the system, we can assume  that $l_j\neq 0$ for some $j\in \{0,\ldots, \ell\}$. If $l_1=\cdots=l_\ell=0$, then $l_0\neq 0$,
and in this case (recall that $n_0=0$)
$$
\int  \prod_{j=0}^\ell T^{pn_j}F_j\, d\mu=0
$$ for every $n_1,\ldots,n_\ell\in\N$, so \eqref{E:l} clearly holds.
Suppose now  that $l_j\neq 0$ for some $j\in\{1,\ldots,\ell\}$. Without loss  of generality we can assume that $l_1\neq 0$.

Note that the limit in \eqref{E:l} is equal to
\begin{equation}\label{E:lim}
\lim_{|n|\to\infty, n\in S} \Big(\E_{p\in\P_d} \int \int  \prod_{j=0}^\ell e\Big(k_j\frac{x+pn_j}{u}\Big) \, e\big(l_j\cdot \big(y+pn_j \alpha\big)\big)\, dm_{\T^v}(y) \, dm_{\Z_u}(x)\Big).
\end{equation}
 For $n=(n_1,\ldots, n_\ell)\in \N^\ell$ let $\beta_n\in\T$ be defined by
$$
\beta_n:=\sum_{j=1}^\ell n_j \, (l_j \cdot \alpha) +\sum_{j=1}^\ell \frac{k_j n_j}{u}.
$$
Since the set $S$ has independent elements, $\alpha=(\alpha_1,\ldots, \alpha_v)$ is such that $1,\alpha_1,\ldots, \alpha_v$ are rationally independent,  and $l_1\neq 0$,
an easy computation shows that $\beta_n$ is irrational
   for all but finitely many $n\in S$.  For those values of $n\in S$, it is well known that the
 sequence  $(p\beta_n)_{p\in \P_d}$ is equidistributed on $\T$. Hence, the bounded convergence theorem gives that the averages $\E_{\P_d}$  in  \eqref{E:lim} are $0$ for all but finitely many $n\in S$, and as a consequence
\eqref{E:l} holds.

Next, we prove  Proposition~\ref{P:ET1}.
Again by  approximation  in $L^2(\mu)$, it suffices to verify \eqref{E:ET1} when $$
F_j(x,y):=e\big(k_j\frac{x}{u}\big) \, e(l_j\cdot y),\quad  k_j\in\Z,\,  l_j\in\Z^v,\,  j=0,\ldots, \ell,
$$
 where $x\in \Z_u$ and $y\in \T^v$.
Suppose first  that at least one of the $l_0,\ldots, l_\ell$ is non-zero. As before, we can assume that $l_1\neq 0$. Then the right hand side in \eqref{E:ET1} is $0$ and the left hand side
 is equal to
\begin{equation}\label{E:lim'}
\E_{n\in\N^r}\E_{p\in\P_d} \int \int \prod_{j=0}^\ell e\Big(k_j\frac{x+pa_j(n)}{u}\Big) \, e\Big(l_j\cdot \big(y+pa_j(n)\alpha\big)\Big)\, dm_{\T^v}(y) \, dm_{\Z_u}(x).
\end{equation}
As before,   we argue that for all $n\in \N^r$ outside a set of density $0$, the averages $\E_{p\in \P_d}$ are $0$ and as a consequence \eqref{E:ET1}  holds.

Suppose now that $l_0=\cdots=l_\ell=0$.
If  all $k_0,\ldots, k_\ell$ are multiples of $u$, then \eqref{E:ET1}  holds trivially. If not, as before, we can assume that  $k_1$ is not a multiple of $u$. Then the right hand side in \eqref{E:ET1} is $0$ and the left hand side
is equal to
$$
\E_{n\in\N^r}\E_{p\in\P_d}  \int \prod_{j=0}^\ell e\Big(k_j\frac{x+pa_j(n)}{u}\Big)  dm_{\Z_u}(x).
$$
If $\sum_{j=0}^\ell k_j$ is not a multiple of $u$, then all the integrals are $0$. Otherwise,
the expression becomes
$$
\E_{n\in\N^r}\E_{p\in\P_d}  \,   e\Big(\sum_{j=1}^\ell k_ja_j(n)\frac{p}{u}\Big).
$$
The average over $p$ is finite (it is equal to the average over those $k\in \{0,\ldots, u-1\}$ that satisfy $(k,u)=1$ and $k\in d\N+1$), hence we can  freely exchange  the two averages and we get the expression
$$
\E_{p\in\P_d}\E_{n\in\N^r}  \,   e\Big(\sum_{j=1}^\ell k_ja_j(n)\frac{p}{u}\Big).
$$
If the last expression is non-zero, then there exists $p\in \{1,\ldots, u-1\}$ such that
$$
\E_{n\in\N^r}  \,   e\Big(\sum_{j=1}^\ell k_ja_j(n)\frac{p}{u}\Big)\neq 0.
$$
Since $k_1$ is not a multiple of $u$, this contradicts our assumption that the sequences $a_1,\ldots, a_\ell$ are jointly equidistributed in congruence classes and completes the proof.
\end{proof}


 The next result follows from Theorem~7 and Corollary~8 of Chapter 15 in \cite{HK18}  (see also \cite[Theorem~5.4]{BHK05}) and
  was  first established in a slightly different form in \cite{Zi05}.
\begin{lemma}\label{L:keyinv}
Let $\ell,   n_1,\ldots, n_\ell\in\N$.      Let $X=G/\Gamma$ be a   nilmanifold  and  $b\in G$ be an ergodic nilrotation. Then for $m_X$-almost every $x\in X$ the sequence
  $$
  (b^{n_1m} x,\ldots,  b^{n_\ell m} x)_{m\in \N}
  $$ is equidistributed on a set $W_x=H\cdot \tilde{x}$, where $\tilde{x}:=(x,\ldots, x)$ and
   $H$ is a subgroup of $G^\ell$  such that
$(g^{n_1},\ldots, g^{n_\ell})\in H$ for every $g\in G$.
\end{lemma}
\begin{remark}
We caution the reader  that although for Abelian nilmanifolds the conclusion holds for every $x\in X$ this is not so for $s$-step nilmanifolds when $s\geq 2$.
\end{remark}
The previous lemma is used in order to establish the following result:
\begin{lemma}\label{L:orthogonality}
		Let  $\ell\in\N$ and 	 $S$ be a subset of $\N^\ell$ with independent elements.
 Let $s\in \N$,  with  $s\geq 2$, $X=G/\Gamma$ be an $s$-step nilmanifold,  $b\in G$ be an ergodic nilrotation, and   $\Phi_1, \ldots, \Phi_\ell$ be  $s$-step nilcharacters of $X$, at least one of which is non-trivial.
   Then for all but finitely many $(n_1,\ldots, n_\ell)\in S$ we have
\begin{equation}\label{E:keypp'}
\E_{m\in\N} \prod_{j=1}^\ell \Phi_j(b^{mn_j}\, x)=0
\end{equation}
for  $m_X$-almost every $x\in X$.
\end{lemma}
\begin{proof}
Henceforth, with $n_1,\ldots, n_\ell$ we denote the coordinates of $n\in \N^\ell$.
By Lemma~\ref{L:keyinv} there exists $X'\subset X$ with  $m_X(X')=1$ such that for every  $x\in X'$ and   every $n\in \N^\ell$ we have
$$
\E_{m\in\N} \prod_{j=1}^\ell \Phi_j(b^{mn_j}\, x)=\int\prod_{j=1}^\ell\Phi_j(x_j)\, dm_{W_{x,n}}(x_1,\ldots, x_\ell),
$$
where $W_{x,n}=H_n\, \tilde{x}$,  $\tilde{x}:=(x,\ldots, x)$, and
   $H_n$ is a subgroup of $G^\ell$  such that
$(g^{n_1},\ldots, g^{n_\ell})\in H_n$ for every $g\in G$. In particular, for every $x\in X'$, we have that
 \begin{equation}\label{E:invariance'}
 (u^{n_1},\ldots,  u^{n_\ell})\cdot W_{x,n}=W_{x,n}   \qquad \text{for every }  n\in \N^\ell \text{ and } u\in G_s.
\end{equation}

Since $s\geq 2$, as remarked in Section~\ref{SS:nil}
 we can assume that $G_s=\T^t$ for some $t\in \N$ ($G_s$ is non-trivial since $X$ supports a non-trivial $s$-step nilcharacter). Henceforth, we use additive notation for elements of $G_s$.
By assumption, for $j=1,\ldots, \ell$ we have
\begin{equation}\label{E:nilchar}
\Phi_j(ux)=e(k_j\cdot u)\, \Phi_j(x), \quad x\in X, u\in G_s=\T^t
\end{equation}
where $k_j\in \Z^t$. Since at least one of the nilcharacters is non-trivial, we can assume that $\Phi_1$ is, or equivalently, that $k_1\in \Z^t$ is non-zero.
Using this, the translation invariance of the measure $m_{W_{x,n}}$,  and \eqref{E:invariance'}, \eqref{E:nilchar},  we deduce that
for every $x\in X'$ and $n\in \N^\ell$ we have
 \begin{equation}\label{E:intzero}
 \int\prod_{j=1}^\ell\Phi_j(x_j)\, dm_{W_{x,n}}(x_1,\ldots, x_\ell)=0
 \end{equation}
 unless $\sum_{j=1}^\ell n_j \, (k_j\cdot u)=0   \pmod{1}$ for every $u\in \T^t$.
 Equivalently, writing $k_j=(k_{j,1},\ldots, k_{j,t})$, $j=1,\ldots, \ell$,  and $u=(u_1,\ldots, u_t)$, we get that if the previous sum is zero, then
 $$
 \sum_{i=1}^tu_i\Big(\sum_{j=1}^\ell n_jk_{i,j}\Big)=0 \! \!  \pmod{1} \quad \text{ for all }
 u_1,\ldots, u_t\in \T.
 $$
 Hence,
 $$
 \sum_{j=1}^\ell n_jk_{i,j}=0 \quad \text{ for } i=1,\ldots, t.
 $$
 Since the integers $k_{1,1},\ldots, k_{1,\ell}$ are not all zero (recall that $k_1\neq 0$), and the  set $S$  has independent elements, we get  that the first identity of the previous system  can be satisfied  only for finitely many $n\in S$.
Hence, for all but finitely many $n\in S$,  for every $x\in X'$ equation \eqref{E:intzero} holds, and as a consequence  \eqref{E:keypp'} holds.  This completes the proof.
\end{proof}

We will also use the following result, which is proved using the Gowers uniformity of the modified von Mangoldt function~\cite{GT10}. It is proved in \cite[Theorem~4.4]{FH18} for $d=1$ but the same argument gives the proof for general $d\in \N$.
\begin{proposition}\label{P:ergid2}
Let $d,r_0\in \N$ and   $(X, \mu, T)$ be a system  such that the ergodic components of the system $(X,\mu,T^{r_0})$ are totally ergodic. Let
$$
A_{d,r_0}:=\big\{m\in \N\colon (m,r_0)=1 \ \text{ and }\  m\equiv 1\! \! \! \pmod{d}\big\}.
$$
Then  all the  limits  below exist and we have
\begin{equation}\label{E:primeform}
		\E_{p\in \P_d} \int  \prod_{j=0}^\ell T^{pj}F_j\, d\mu= \E_{m\in A_{d,r_0}} \int \prod_{j=0}^\ell T^{mj}F_j\, d\mu
	\end{equation}
	for all $\ell \in \N$ and $F_0,\ldots, F_\ell\in L^\infty(\mu)$.
\end{proposition}

\begin{proof}[Proof of Propositions~ \ref{P:ET2} and \ref{P:ET1}]
	By Propositions~\ref{P:red1} and \ref{P:red2} we can assume that the system is an ergodic nilsystem.
	So let $X=G/\Gamma$ be an $s$-step nilmanifold with the Haar measure $m_X$ and $Tx=bx$, $x\in X$, for some $b\in G$.
	
		We prove the statements by induction on $s\in \N$ .
If $s=1$, then $X$ is a compact Abelian Lie group so we are covered  by Lemma~\ref{L:Abelian}.

Suppose that $s\geq 2$ and Propositions~ \ref{P:ET2} and \ref{P:ET1} hold for all $(s-1)$-step nilsystems. Since linear combinations of $s$-step nilcharacters are dense in $C(X)$, using an approximation argument, we can assume that
for $j=0,\ldots, \ell$ we have $F_j=\Phi_j$ where $\Phi_j$ is an $s$-step nilcharacter of $X$.
If   $\Phi_0,\ldots, \Phi_\ell$ are all   trivial $s$-step nilcharacters of $X$, then they factorize through the nilmanifold
$$
X':=G/(G_s\Gamma)=(G/G_s)/((\Gamma\cap G_s)/G_s).
$$
The group $G/G_s$ is $(s-1)$-step nilpotent and $X'$ is an $(s-1)$-step nilmanifold.  So in this case the result follows from the induction hypothesis. Hence, we can assume that at least one of the $s$-step nilcharacters
$\Phi_0,\ldots, \Phi_\ell$ is non-trivial.

Suppose first that  $\Phi_0$ is a non-trivial $s$-step nilcharacter and $\Phi_1,\ldots,\Phi_\ell$ are trivial $s$-step nilcharacters.  Then for every  $n_0,\ldots, n_\ell \in\Z$ the function
$\prod_{j=0}^\ell T^{pn_j}\Phi_j$ is  also a non-trivial $s$-step nilcharacter (with the same frequency as $\Phi_0$). Hence, $\int\Phi_0(x)\, dm_X=0$ and
$$
\int \prod_{j=0}^\ell \Phi_j(T^{pn_j}x)\, dm_X=0 \quad \text{ for every } n_0,\ldots, n_\ell \in\Z.
$$
 In this case, equations \eqref{E:l} and \eqref{E:ET1}  clearly hold.

Hence,  without loss of generality, we can assume that  $\Phi_1$ is a non-trivial $s$-step nilcharacter. In this case we
 prove  Proposition~\ref{P:ET2}, the proof of Proposition~\ref{P:ET1} is very similar. We have to show that
$$
\lim_{|n|\to \infty, n\in S}\Big( \E_{p\in\P_d}\int \prod_{j=0}^\ell \Phi_j(T^{pn_j}x)\, dm_X\Big)=0.
$$
Since $(X,\mu,T)$ is an ergodic nilsystem,  there exists $r_0\in \N$ such that $T^{r_0}$ acts ergodically on $X^0$ and  the ergodic components of $(X,\mu,T^{r_0})$ are totally ergodic (see for example Corollaries 7 and 8 on page 182 in \cite{HK18}).
By Proposition~\ref{P:ergid2} it
suffices to show that for every $k\in\N$ we have
$$
\lim_{|n|\to \infty, n\in S}\Big(\E_{m\in\N} \int \prod_{j=0}^\ell \Phi_j(T^{(dr_0m+k)n_j}x)\, dm_X\Big)=0.
$$

 Recall that $n_0=0$. By the bounded convergence theorem (from  \cite{L05} the limits $\E_{m\in\N}$  exist for every $x\in X$) it suffices to show that for every $k\in\N$, for $m_X$-almost every $x\in X$, we have
$$
\lim_{|n|\to \infty, n\in S} \Big(\E_{m\in\N} \prod_{j=1}^\ell \Phi_j(T^{(dr_0m+k)n_j}x)\Big)=0.
$$
Hence, it suffices to show that for every $k\in\N$, for $m_X$-almost every $x\in X$, we have
$$
\lim_{|n|\to \infty, n\in S} \Big(\E_{m\in\N} \prod_{j=1}^\ell \Phi'_j(b'^{mn_j}\cdot x)\Big)=0,
$$
where $b':=b^{dr_0}$ and for $j=1,\ldots, \ell$ we let $\Phi_j'(x):=\Phi_j(b^{kn_j}x)$ for  $x\in X$. Then $b'$ acts ergodically on $X_0$  and as explained in Section~\ref{SS:nil}, for $j=1,\ldots, \ell$,  the restriction of $\Phi'_j$ to $X_0$ is  an $s$-step  nilcharacter of $X_0$ with same frequency as $\Phi_j$; hence, $\Phi_1$  is non-trivial. The validity of the last
identity  follows from Lemma~\ref{L:orthogonality} (here we used crucially that $s\geq 2$). This completes the proof.
\end{proof}

\end{document}